\documentclass[11pt,a4paper]{article} 
\usepackage{amsmath,amssymb,amsthm,amsfonts}
\usepackage{enumerate,color,bm}
\usepackage[utf8]{inputenc}
\usepackage[T1]{fontenc}
\usepackage{geometry}
\numberwithin{equation}{section} 
\newtheorem{thm}{Theorem}[section]
\newtheorem{corollary}[thm]{Corollary}
\newtheorem{lem}[thm]{Lemma}

\theoremstyle{definition}
\newtheorem{df}[thm]{Definition}
\newtheorem{remark}[thm]{Remark}

\newcommand{\ep}{\varepsilon}
\newcommand{\eps}{\varepsilon}
\newcommand{\pa}{\partial}
\newcommand{\RN}{\mathbb{R}^N}
\newcommand{\nep}{n_{\ep}}
\newcommand{\nept}{n_{\ep t}}
\newcommand{\cep}{c_{\ep}}
\newcommand{\cept}{c_{\ep t}}
\newcommand{\uep}{u_{\ep}}
\newcommand{\uept}{u_{\ep t}}
\newcommand{\Rone}{\mathbb{R}}
\newcommand{\R}{\mathbb{R}}
\newcommand{\f}[2]{\frac{#1}{#2}}
\newcommand{\na}{\nabla}

\newcommand{\norm}[2][]{\left\|#2\right\|_{#1}}
\newcommand{\Lom}[1]{L^{#1}(\Omega)}

\newcommand{\Ombar}{\overline{\Omega}}
\newcommand{\Om}{\Omega}

\newcommand{\tmax}{T_{{\rm max}, \ep}}
\newcommand{\lp}[2]{\|#2\|_{L^{#1}(\Omega)}}
\newcommand{\biglp}[2]{\big\|#2\big\|_{L^{#1}(\Omega)}}

\newcommand{\ol}{\overline}

\newcommand{\HP}{\mathcal{P}}

\newcommand{\cd}{(\cdot,t)}

\newcommand{\io}{\int_{\Omega}}
\newcommand{\nn}{\nonumber}

\newcommand{\iio}{\int_0^T \int_\Omega}
\newcommand{\iiio}{\int_0^\infty \int_\Omega}
\newcommand{\nepp}{\nep^p}
\newcommand{\cepq}{\cep^q}

\begin{document}
\begin{center}
    \LARGE{{\bf 
 A Keller--Segel--fluid system with singular sensitivity: Generalized solutions}}
\end{center}
\vspace{5pt}
\begin{center}
    Tobias Black\\
    \vspace{2pt}
    Universit\"at Paderborn, 
    Institut f\"ur Mathematik,\\ 
    Warburger Str.\ 100, 33098 Paderborn, Germany\\
    {\tt tblack@math.uni-paderborn.de}\\
    \vspace{12pt} 
    Johannes Lankeit\\
    \vspace{2pt}
    Universit\"at Paderborn, 
    Institut f\"ur Mathematik,\\ 
    Warburger Str.\ 100, 33098 Paderborn, Germany\\
    {\tt jlankeit@math.uni-paderborn.de}\\
    \vspace{12pt}
    Masaaki Mizukami\\
    \vspace{2pt}
    Department of Mathematics, 
    Tokyo University of Science\\
    1-3, Kagurazaka, Shinjuku-ku, Tokyo 162-8601, Japan\\
    {\tt masaaki.mizukami.math@gmail.com}\\
    \vspace{2pt}
\end{center}
\begin{center}    
    \small \today
\end{center}

\vspace{2pt}
\newenvironment{summary}
{\vspace{.5\baselineskip}\begin{list}{}{%
     \setlength{\baselineskip}{0.85\baselineskip}
     \setlength{\topsep}{0pt}
     \setlength{\leftmargin}{12mm}
     \setlength{\rightmargin}{12mm}
     \setlength{\listparindent}{0mm}
     \setlength{\itemindent}{\listparindent}
     \setlength{\parsep}{0pt}
     \item\relax}}{\end{list}\vspace{.5\baselineskip}}
\begin{summary}
{{\bf Abstract.}
In bounded smooth domains $\Omega\subset\mathbb{R}^N$, $N\in\{2,3\}$, we consider the Keller--Segel--Stokes system 
\begin{align*}
 n_t + u\cdot \nabla n &= \Delta n - \chi \nabla \cdot(\frac{n}{c}\nabla c),\\
 c_t + u\cdot \nabla c &= \Delta c - c + n,\\
 u_t &= \Delta u + \nabla P + n\nabla \phi, \qquad \nabla \cdot u=0,
\end{align*}
and prove global existence of generalized solutions if 
\[
 \chi<\begin{cases}
       \infty,&N=2,\\
       \frac{5}{3},&N=3.
      \end{cases}
\]
These solutions are such that blow-up into a persistent Dirac-type singularity is excluded. 
\vspace*{1cm}
}
\end{summary}
\vspace{10pt}

\newpage
%
%
\section{Introduction}
If chemotaxis takes place in a fluid environment, it seems reasonable to include interaction with the sourrounding fluid into the model; in particular, since experiments indicate that in the regime of a high number of chemotactic agents this interaction ceases to be negligible (cf.\ \cite{dombrowskietal}). 
The mathematical question that immediately arises is the query to which extent the presence of this coupling affects properties of the solution or the proofs thereof. In some sense, this can be understood as question about indirect regularity effects of a fluid flow. 

In this article we are going to consider this question in the setting of a chemotaxis system with singular sensitivity: 
 \begin{align}\label{introsystem}
        & n_t + u\cdot\nabla n 
         = \Delta n - \chi \nabla\cdot \left(\frac{n}{c}\nabla c\right),  \nn
 \\[1mm]
         &c_t + u\cdot\nabla c 
         = \Delta c - c + n,  
 \\[1mm]
        &u_t 
        = \Delta u + \nabla P 
        + n \nabla\phi, 
        \quad \nabla\cdot u = 0,\nn
\end{align}
where $n$, $c$, $u$ respectively denote the density of chemotactically active bacteria, the concentration of a signal substance and the velocity field of the fluid, whose motion is driven by density differences according to presence or absence of bacteria.

%
%

In the presence of fluid coupling, we have been able to obtain global existence of classical solutions for $\chi<\sqrt{\f2N}$ in \cite{blm2}. (We also refer to the introduction of said article for additional motivation and more references to works dealing with chemotaxis--fluid systems or chemotaxis systems with logarithmic sensitivity.) 

This parameter range for $\chi$ is (almost) as large as known for the fluid-free system (cf.\  \cite{biler99,nagai_senba_yoshida_ge,win_singular,fujie})  -- there it is only known to be slightly larger in $N=2$, cf.\ \cite{lankeit_m2as}. 
Beyond this range, weaker solution concepts have been explored, excluding at least the possibility of blow-up into a persistent Dirac-type singularity, \cite{win_singular,stinner_win,lankeit_winkler}, whereas  blow-up can be expected for large values of $\chi$, according to the result of \cite{nagai_senba} on the corresponding parabolic--elliptic system. 


While for small $\chi$, the proofs of global existence of classical solutions (see \cite{win_singular}) and even boundedness, \cite{fujie}, rely on an ODI for $\io n^pc^{-r}$ for some $p>1$ and suitable $r$, the decisive estimates for the construction of generalized solutions for larger $\chi$ in \cite{lankeit_winkler} are based on a similar observation concerning $\io n^pc^{-r}$ for $p$ below $1$. 

It turned out that corresponding estimates allow for a proof of a supersolution property involving the compound quantity $n^pc^{-r}$ with $p,-r\in (0,1)$, which if combined with a more common notion of weak solubility for the second equation and with the condition that the mass $\io n$ be nonincreasing (as a faint subsolution requirement) serves to yield a solution concept which is compatible with the usual concept, but can cope with much less regularity information, and has successfully been employed in systems where the existence of global solutions of any kind had been unknown (\cite{lankeit_winkler} and, in a parabolic-elliptic setting, \cite{black-pe}). Up to now, however, the treatments of this approach do not extend to any fluid-coupled systems. It is, therefore, aim of the present article to expand said technique to the fluid context.

In order to see how, in the latter setting, reliance on estimates for $\io n^pc^{-r}$ presents us with a problem, let us recall the main difficulty stemming from presence of the fluid coupling in \cite{blm2}: 

If we consider the second equation in \eqref{introsystem} as inhomogeneous heat equation $c_t=\Delta c - c +f$, due to the transport term we not only lose positivity information on the source (we knew the sign of $n$ but have no information on that of $n-u\nabla c$), important for the global boundedness proof, but, more crucially, also bounds enabling us to employ heat semigroup estimates directly: Where the usual mass conservation of the first equation readily yielded an $L^1$-bound if $u\equiv 0$ and hence $f\equiv n$, at the beginning we are lacking comparable estimates for $f\equiv n-u\nabla c$.

In \cite[Lemma 2.5]{blm2}, we mitigated this problem by replacing the use of semigroup estimates by an argument based on the differential inequality (\cite[Lemma 2.4]{blm2})
\[
 \f1q \f{d}{dt} \io c^q = -(q-1) \io c^{q-2}|\na c|^2 - \io c^q + \io nc^{q-1} 
\]
for arbitrary $q>1$, where we were able to control the source term mainly due to the bound on $\io n^pc^{-r}$ previously obtained, \cite[Lemma 2.3]{blm2}. This will no longer be possible if $p<1$. 

We work around this restriction in different ways for $N=2$ and $N=3$. In the two-dimensional setting, we firstly procure bounds for the fluid velocity field 
and then rely on the well-known smoothing estimates for the heat semigroup; however, we need more than a straightforward application and have to partially absorb the additional source term by the term to be controlled. 

Unfortunately, this reasoning fails for $N=3$ (cf.\ Remark \ref{remark:doesntworkforNequal3}). 
Here we instead employ a differential inequality for $\f{d}{dt} \io c^q$ -- for $q<1$, in contrast to \cite{blm2}. One of its consequences is a bound on the space-time integral of $|\na c^{\f q2}|^2$ (Lemma \ref{lem;L2;nablac^q/2}), which we then use to secure bounds for $\int_0^T\io c^r$ in Lemma \ref{lem;Lresti;cep;N=3} for $r\in(1,\f53)$. (This reasoning, in turn, would work for $N=2$, but entail some restrictions on $\chi$.)

If we want to control $\int_0^T\io n^{\rho}$ for some $\rho>1$ (which is crucial not only for some of the convergence results in Lemma \ref{lem;conv}, but also for obtaining the minimal regularity we desire for our solutions if they are meant to exclude blow-up into a persistent Dirac-type singularity), as in Lemma \ref{lem;Lr(Lr;n)}, the restriction $r<\f53$ will force us (cf.\ \eqref{ineq;L1-qr..;nr}) to pose a stronger condition on $\chi$ in the $3$-dimensional case than was needed in the fluid-free setting in \cite{lankeit_winkler}. 

Aside from these complications, however, it is possible to adapt the solution concept of \cite{lankeit_winkler} to the present system. We introduce generalized solutions in Section \ref{sec:solutionconcept} and then, roughly following the reasoning of  \cite{lankeit_winkler} with the changes indicated above -- and, of course, additional modifications whenever the presence of fluid terms demands them --, show the global existence of generalized solutions to \eqref{introsystem}. 

More precisely, we will assume that the initial data and parameter in 
\begin{subequations}\label{cp}
 \begin{align}\label{cp1}
        & n_t + u\cdot\nabla n 
         = \Delta n - \chi \nabla\cdot \left(\frac{n}{c}\nabla c\right)
         &&\text{in } \Omega\times(0,T),  
 \\[1mm]\label{cp2}
         &c_t + u\cdot\nabla c 
         = \Delta c - c + n 
         && \text{in } \Omega\times(0,T),  
 \\[1mm]\label{cp3}
        &u_t 
        = \Delta u + \nabla P 
        + n \nabla\Phi, 
        \quad \nabla\cdot u = 0 
        && \text{in } \Omega\times(0,T)
 \\[1mm] \label{cp4}
         & \partial_\nu n =
        \partial_\nu c = 0, \quad 
        u = 0 
        && \text{in }\partial\Omega\times(0,T),
 \\[1mm] \label{cp5}
        &    n(\cdot,0)=n_0,\ c(\cdot,0)=c_0,\ 
        u(\cdot,0)=u_0 
        &&  \text{in } \Omega, 
 \end{align}
\end{subequations} 
satisfy $\chi>0$ and  
\begin{equation}\label{condi;ini2}
 \Phi \in C^{2}(\overline{\Omega}). 
\end{equation}
as well as 
 \begin{align}\label{condi;ini1}
   &0 \le n_{0} 
   \in C^0(\overline{\Omega}), \quad n_0 \not\equiv 0,
  \\ 
   &c_0 \in W^{1,\infty}(\Omega), 
   \quad \inf_{x\in \Omega}c_0(x) > 0,\label{condi;ini3}  
  \\ 
   &u_0 \in D(A^{\alpha}), \label{condi;ini4}
 \end{align}
for some $\alpha \in \left(\frac{N}{4}, 1\right)$, where $A:=-\HP\Delta$ denotes the Stokes operator, with Helmholtz projection $\HP$ onto the subspace 
$L^2_\sigma\!\left(\Omega\right):=\big\{\varphi\in L^2\big(\Omega;\Rone^N\big)\,\vert\,\nabla\cdot\varphi=0\ \mbox{in}\ \Omega \big\}$
and homogeneous Dirichlet boundary conditions.

The main result of this article will then be given by: 

\begin{thm}\label{mainthm}
For $N\in\{2,3\}$ let $\Omega\subset \mathbb{R}^N$ 
 be a bounded domain with smooth boundary. 
 Suppose that $\Phi,n_{0},c_0,u_0$ fulfil 
 \eqref{condi;ini2}{\rm --}\eqref{condi;ini4} 
 and $\chi > 0$ satisfies
 \begin{align}\label{condi;chi;generalized}
 \chi < 
 \left\{ 
   \begin{array}{ll} \infty\quad  & \mbox{if}\ N=2, 
 \\[1.5mm] 
   \frac 53 & \mbox{if} \ N=3. 
   \end{array}
 \right. 
 \end{align}
 Then there exist at least one global generalized solution $(n,c,u)$ in the sense of 
 Definition  \ref{def;generalizedsol}. 
 In particular, this solution satisfies $n \in L^s_{loc} (\ol{\Omega}\times [0,\infty))$ 
 for some $s>1$, and moreover we have 
\[
\io n \cd = \io n_0  \quad\mbox{for a.e. } t>0. 
\]
\end{thm}

\section{Generalized solutions}\label{sec:solutionconcept}

In this section we adapt the definition of generalized solvability from \cite{lankeit_winkler} to also incorporate fluid interaction.
As a first step let us introduce the notion of {\it global weak $(p,q)$-supersolutions} to
\eqref{cp1} and \eqref{cp4}--\eqref{cp5}. 

\begin{df}\label{def;weak(p,q)supersol}
Let $p\in (0,1)$ and $q\in (0,1)$, and suppose that 
$n,c:\Omega \times (0,\infty)\to \mathbb{R}$ and $u:\Omega\times (0,\infty)\to\mathbb{R}^N$ 
are measurable functions 
on $\Omega \times (0,\infty)$ such that $n>0$ and $c>0$ a.e. in $\Omega\times (0,\infty)$,
that 
\begin{align}\label{def;supersol;regularity1}
  n^p c^q \in L^{\frac 32+\eta}_{loc}(\ol{\Omega}\times [0,\infty)),
  \quad n^{p+1} c^{q-1} \in L^1_{loc}(\ol{\Omega}\times [0,\infty)) \quad \mbox{and} 
  \quad u\in  L^{\frac{3+2\eta}{1+2\eta}}_{loc}([0,\infty);L^{\frac{3+2\eta}{1+2\eta}}_{\sigma}(\Omega))
\end{align}
with some $\eta>0$ and that $\nabla n^{\frac p2}$ and $\nabla c^{\frac q2}$ belong to 
$L^1_{loc} (\Omega \times (0,\infty);\Rone^N)$ and are such that 
\begin{align}\label{def;supersol;regularity2}
  c^{\frac q2} \nabla n^{\frac p2}\in L^2_{loc}(\ol{\Omega}\times [0,\infty);\Rone^N) 
  \quad \mbox{and} \quad 
  n^{\frac p2} \nabla c^{\frac q2} \in L^2_{loc}(\ol{\Omega}\times [0,\infty);\Rone^N).  
\end{align}
Then $(n,c,u)$ will be called a global weak $(p,q)$-supersolution of 
\eqref{cp1} and \eqref{cp4}--\eqref{cp5} 
if 
\begin{align}\label{ineq;def;weak(p,q)}\notag
  - \iiio n^pc^q\varphi_t - \io n_0^p c_0^q \varphi(\cdot,0) 
  & \ge  \frac{4(1-p)q -4q^2 -p(1-p)^2\chi^2}{pq(p\chi +1-q)} \iiio c^q |\nabla n^{\frac p2}|^2\varphi  
\\ \notag  
  &\quad\, + \frac{4(p\chi + 1-q)}{q} \iiio \left| n^{\frac p2} \nabla c^{\frac q2}  
      -  \frac{(1-p)\chi + 2q}{2(p\chi +1-q)}c^{\frac q2} \nabla n^{\frac p2}\right|^2 \varphi 
\\  \notag
  &\quad\, -\frac{2p\chi}{q} \iiio n^{\frac p2}c^{q} \nabla n^{\frac p2}\cdot \nabla\varphi 
   + \left( 1-\frac{p\chi}{q} \right) \iiio n^p c^q \Delta \varphi 
\\ \notag
  &\quad\, -q\iiio n^p c^q \varphi + q\iiio n^{p+1}c^{q-1} \varphi  
\\
&\quad\, + \iiio n^p c^q u\cdot \nabla \varphi 
\end{align}  
holds for all nonnegative $\varphi \in C^\infty_0(\ol{\Omega}\times [0,\infty))$  
such that $\frac{\partial \varphi}{\partial \nu}=0$ on $\partial \Omega \times (0,\infty)$ and 
if moreover 
\[
  n^p c^q >0 \quad \mbox{a.e. on}\ \partial \Omega \times (0,\infty). 
\]
\end{df} 
\begin{remark} 
The regularity requirements ensure that 
actually all of the integrals appearing in \eqref{ineq;def;weak(p,q)} exist; and 
since $u^{\frac p2}c^{\frac q2}\in L^2_{loc}([0,\infty);W^{1,2}(\Omega))\hookrightarrow 
L^2_{loc}(\pa \Omega \times [0,\infty))$, also the positivity condition 
at the boundary makes sense.  
\end{remark}

We now provide a definition of {\it weak solutions} to \eqref{cp2}--\eqref{cp5}. 

\begin{df}\label{def;weaksol}
A triplet $(n,c,u)$ of functions satisfying 
\begin{align*}
  \left\{ 
   \begin{array}{l}
    n\in L^1_{loc} (\ol{\Omega}\times [0,\infty)), 
  \\[1mm]
    c\in L^{\frac 32+\eta}_{loc}(\ol{\Omega}\times [0,\infty))\cap L^1_{loc}([0,\infty);W^{1,1}(\Om)), 
  \\[1mm]
     u\in L^{\frac{3+2\eta}{1+2\eta}}_{loc}([0,\infty);L^{\frac{3+2\eta}{1+2\eta}}_{\sigma}(\Omega))
   \end{array}  
  \right.
\end{align*}
for some $\eta>0$ 
will be named a global weak solution of 
\eqref{cp2}--\eqref{cp5} 
if 
\begin{align*}
  -\iiio c\varphi_t -\io c_0\varphi(\cdot,0) = 
  - \iiio \nabla c\cdot \nabla \varphi - \iiio c\varphi + \iiio n\varphi -\iiio cu\cdot\nabla \varphi
\end{align*}
is valid for all $\varphi \in C^\infty_0(\ol{\Omega}\times(0,\infty))$, and 
\begin{align*}
  - \iiio u\cdot \psi_t - \io u_0\cdot \psi(\cdot,0) = \iiio \nabla u\cdot\nabla \psi - \iiio n \nabla\Phi \cdot \psi
\end{align*}
is satisfied for all $\psi\in C^\infty_{0,\sigma}(\ol{\Omega}\times [0,\infty)):=\{\psi\in C_0^\infty(\ol{\Omega}\times [0,\infty);\Rone^N)\mid \nabla\cdot \psi=0\}$. 
\end{df} 

Finally we introduce a definition of {\it global generalized solutions} to \eqref{cp} as follows. 

\begin{df}\label{def;generalizedsol}
A triplet of measurable functions $n,c$ and $u$ 
defined on $\Omega \times (0,\infty)$ will be said to be a global generalized solution of 
\eqref{cp} if $(n,c,u)$ is a global weak solution of 
\eqref{cp2}--\eqref{cp5} 
according to Definition \ref{def;weaksol}, if there exist $p\in (0,1)$ and $q\in (0,1)$ such that 
$(n,c,u)$ is a global weak $(p,q)$-supersolution of 
\eqref{cp1} and \eqref{cp4}--\eqref{cp5} 
in the sense of Definition \ref{def;weak(p,q)supersol}, and if moreover 
\begin{equation*}
  \io n(\cdot,t) \le \io n_0 \quad \mbox{for a.e.} \ t>0. 
\end{equation*}
\end{df}

\begin{remark}
 If $(n,c,u)\in \left(C^0(\Ombar\times[0,\infty))\cap C^{2,1}(\Ombar\times(0,\infty))\right)^{2+N}$ is a global generalized solution to \eqref{cp}, then $(n,c,u)$ solves \eqref{cp} classically. For a proof in the fluid-free setting, see \cite[Lemma 2.5]{lankeit_winkler}. In this proof it can also be seen why Definition \ref{def;weak(p,q)supersol} includes positivity requirements on $n$ and $c$, both in the domain and on the boundary.
\end{remark}

\section{Properties and global existence of classical solutions to a family of approximate problems}
In this section we investigate a family of approximate problems and derive basic solution properties, which on one hand act as starting point for further a priori bounds and on the other hand allow us to conclude that, in fact, these solutions are global-in-time.
For $\ep\in(0,1)$ we will make use of a convenient regularization of 
\eqref{cp1}--\eqref{cp5}, by considering  
\begin{subequations}
 \begin{align}\label{apcp1}
         &\nept + \uep \cdot\nabla \nep  
         = \Delta \nep - \chi \nabla\cdot\Big(\frac{\nep}{(1+\ep\nep)\cep}\nabla \cep\Big)
         && \text{in } \Omega\times(0,T), 
 \\[1mm]\label{apcp2}
         &\cept + \uep\cdot\nabla \cep 
         = \Delta \cep - \cep + \nep 
         && \text{in } \Omega\times(0,T), 
 \\[1mm]\label{apcp3}
        &\uept 
        = \Delta \uep + \nabla P_\ep 
        + \nep \nabla\Phi, 
        \quad \nabla\cdot \uep = 0 
        && \text{in } \Omega\times(0,T), 
\\[1mm]\label{apcp4}
        &   \partial_\nu \nep =
        \partial_\nu \cep = 0, \quad 
        \uep = 0 
        && \text{in }\partial\Omega\times(0,T),
 \\[1mm] \label{apcp5}
        &    \nep (\cdot,0)=n_0,\ \cep (\cdot,0)=c_0,\ 
        \uep (\cdot,0)=u_0 
        &&  \text{in } \Omega.  
 \end{align}
 \end{subequations}
%
%
%
%
We first recall a local existence result. 
We also give some lower estimate for $\cep$, 
which 
will alleviate the difficulties linked to the presence of the 
singular sensitivity function. 

\begin{lem}\label{lem;local existence}
Let $N\in\{2,3\}$, $\chi>0$,  $\alpha\in(\frac{N}{4},1)$ and $\ep\in(0,1)$ and 
let $\Omega\subset \RN$ be a 
bounded domain with smooth boundary. 
Assume that $n_0,c_0,u_0,\Phi$ satisfy 
\eqref{condi;ini2}{\rm --}\eqref{condi;ini1}. 
Then there exist $\tmax\in (0,\infty]$ and 
a classical solution $(\nep,\cep,\uep,P_\ep)$ of \eqref{apcp1}--\eqref{apcp5}  
in $\Omega\times (0,\tmax)$ such that 
\begin{align*}
  &\nep \in C^0 (\overline{\Omega}\times [0,\tmax))\cap 
  C^{2,1}(\overline{\Omega}\times (0,\tmax)), 
  \\ 
  &\cep \in C^0 (\overline{\Omega}\times [0,\tmax))\cap 
  C^{2,1}(\overline{\Omega}\times (0,\tmax))
  \cap L^\infty_{loc}([0,\tmax);W^{1,\infty}(\Omega)), 
  \\ 
  &\uep \in C^0 (\overline{\Omega}\times [0,\tmax);\R^N)\cap 
  C^{2,1}(\overline{\Omega}\times (0,\tmax);\R^N),\\
  &P_\ep \in C^{1, 0}(\overline{\Omega}\times(0, \tmax))  
\end{align*}
and 
\begin{align*}
  \tmax=\infty \quad \mbox{or}\quad 
  \lim_{t\to \tmax}
  \left(
  \lp{\infty}{\nep (\cdot,t)}
  + \|\cep (\cdot,t)\|_{W^{1,\infty}(\Omega)}
  +\lp{2}{A^\alpha \uep (\cdot,t)}
  \right)=\infty. 
\end{align*}
The solution is unique, 
up to addition of a spatially constant function to $P_\ep$ and, moreover, has the properties 
\begin{align}\notag
  &\nep (x,t)\ge 0 
  \quad \mbox{and}
  \\\label{ineq;lower;c} 
  &\cep (x,t)\ge 
  \left(\inf_{y\in{\Omega}}c_0(y)\right)e^{-t}
 \quad \mbox{for all}\ x\in \Om \;\text{ and } \ t\in (0,\tmax). 
\end{align} 
\end{lem}
\begin{proof}
Well-known fixed point arguments, often used in chemotaxis systems (see e.g. 
\cite[Lemma 3.1]{BBTW} 
and 
\cite[Lemma 2.1]{win_CTNS_global_largedata}), 
can be adapted in a similar way as in 
\cite[Theorem 2.3 (i)]{lankeit_m2as}
so as to compensate for the singular sensitivity present in our setting. With these necessary adjustments, the local existence and uniqueness result can be obtained in a straightforward manner. The estimates in \eqref{ineq;lower;c} are direct consequences of the comparison principle.
\end{proof}

In the following, we will always assume that 
$N$, $\Om$, $\chi$, $n_0$, $c_0$, $u_0$, $\Phi$ and $\alpha$ 
obey the conditions of Lemma \ref{lem;local existence} and are fixed. 
For given $\ep\in(0,1)$, by $(\nep,\cep,\uep,P_\ep)$ we will denote the corresponding solution 
to \eqref{apcp1}--\eqref{apcp5} given by Lemma \ref{lem;local existence} and 
by $\tmax$ its maximal existence time. 
Let us continue with some elementary inequalities 
for $\nep$ and $\cep$. 
%

%
%
\begin{lem}\label{lem;L1;nepcep}
For all $\eps \in(0,1)$, 
\[
 \io \nep \cd =\io n_0 \qquad \text{for all } t\in(0,\tmax);
\]
and there is $C>0$ such that for any $q\in(0,1]$ 
\[
 \io \cep^q\cd \le C\qquad \text{for all } t\in(0,\tmax) \text{ and } \eps\in(0,1).
\]
%
\end{lem} 
\begin{proof}
The first part of the lemma  and existence of $C_1>0$ such that $\io \cep\cd \le C_1$ for all $t\in(0,\tmax)$ and $\eps\in(0,1)$ result from integration of 
\eqref{apcp1} and \eqref{apcp2}  
due to 
$\na\cdot \uep =0$ in $\Om\times(0,\tmax)$ and \eqref{apcp4}-\eqref{apcp5}. 
The second part is an immediate consequence, since for any $q\in(0,1]$ and any $\eps\in(0,1)$
\begin{equation*}
  \io \cep^q \le |\Omega|^{1-q} \left(\io \cep\right)^{q} \le |\Omega|^{1-q}C_1^q\le (1+|\Om|)(1+C_1) \qquad \text{in } (0,\tmax).\qedhere
\end{equation*}
\end{proof}

The following well-known result links the regularity of $\uep$ to the known regularity of $\nep$. 

\begin{lem}\label{lem;regurality;u}
For all $p\ge 1$,   
\[
\begin{cases}
r \in [1,\frac{Np}{N-p}) & \mbox{if} \ p\le N, \\
r \in [1,\infty]& \mbox{if} \ p>N,
\end{cases}
\quad \mbox{and} \quad 
\begin{cases}
q \in [1,\frac{Np}{N-2p}) & \mbox{if} \ p\le \frac{N}{2}, \\
q \in [1,\infty]& \mbox{if} \ p>\frac{N}{2}, 
\end{cases}
\] 
there exist $C_r>0$ and $C_q>0$ such that 
for all $M>0$ and all $\ep\in (0,1)$ the following holds\/{\rm :} 
If   
\begin{align*}
\lp{p}{\nep \cd} \le M \quad \mbox{for all}\ t\in (0,\tmax), 
\end{align*}
then 
\[
  \|\uep\cd\|_{W^{1,r}(\Omega)} \le C_r (1+M)
\quad \mbox{and} \quad
  \lp{q}{\uep\cd}\le C_q (1+M)
  \quad \mbox{for all}\ t\in (0,\tmax). 
\] 
\end{lem}
\begin{proof}
The outcome of this lemma can be achieved by utilizing regularity estimates for the Stokes semigroup and embedding properties for domains of fractional powers of the Stokes operator (cf. e.g. \cite[Lemma 2.3]{cao_lankeit}). Detailed proofs can be found in \cite[Lemmata 2.4 and 2.5]{wang_xiang} ($N=2$) and \cite[Corollary 3.4]{win_ctfluid3dnonlineargeneral}
 ($N=3$). 
\end{proof}

Then Lemma \ref{lem;regurality;u} together with the $L^1$-estimate for $\nep$ tells us the following estimates. 

\begin{corollary}\label{lem;Lr;u}
For all $r\in [1,\frac{3}{2})$ and $q\in [1,3)$, there are $C_r>0$ and $C_q>0$ such that 
\[
  \|\uep\cd\|_{W^{1,r}(\Omega)} \le C_r 
\quad \mbox{and} \quad
  \lp{q}{\uep\cd}\le C_q 
\]
hold for all $t\in (0,\tmax)$ and $\ep\in (0,1)$. 
\end{corollary}
\begin{proof}
Combination of Lemmata \ref{lem;L1;nepcep} and \ref{lem;regurality;u} implies this lemma. 
\end{proof}

As last preparatory step for the proof of global-in-time solutions to the approximate system, we shall show the following lemma.

\begin{lem}\label{lem;L2;n;ap}
  For each $\ep\in (0,1)$ and all $T\in (0,\infty)$ there is $C_\ep (T)>0$ such that 
  \[
    \lp{2}{\nep\cd} + \lp{2}{\cep\cd} \le C_\ep(T)   
  \]
holds for all $t\in (0,\min\{T,\tmax\})$. 
\end{lem}
\begin{proof}
Given $\ep\in(0,1)$, we test \eqref{apcp1} by $\frac{1}{2} \nep$ and make use of the lower bound for $c_\ep$ established in \eqref{ineq;lower;c} of Lemma \ref{lem;local existence} to estimate  
$\chi\frac{\nep}{(1+\ep\nep)\cep} \le \frac{\chi e^T}{\ep \inf c_0}=: C_1>0$  
and derive that  
\begin{align}\label{ineq;L2;n;ap}
  \frac 12 \frac{d}{dt} \io \nep^2 
&=  -\io |\nabla \nep|^2 + \chi \io \frac{\nep}{(1+\ep\nep)\cep}\nabla \nep\cdot \nabla \cep 
\le -\frac 12 \io|\nabla \nep|^2 + \frac{C_1^2}{2} \io |\nabla \cep|^2 
\end{align}
in $(0,\min\{T,\tmax\})$. 
Similarly, \eqref{apcp2} implies 
\begin{align}\label{ineq;L2;c;ap}
  \frac 12 \frac{d}{dt} \io \cep^2 
&= -\io |\nabla \cep|^2 + \io \nep\cep - \io \cep^2 
 \le -\io |\nabla \cep|^2 +\frac 12 \io \nep^2 -\frac 12 \io \cep^2
\end{align}
in $(0,\min\{T,\tmax\})$. 
Thus fixing $\kappa > \frac{C_1^2}{2}$, 
we can see from \eqref{ineq;L2;n;ap} and \eqref{ineq;L2;c;ap} that 
\begin{align*}
  \frac 12 \frac{d}{dt} \left( \io \nep^2 +\kappa \io \cep^2\right)
& \le \frac{\kappa}{2} \io \nep^2 - \frac{\kappa}{2} \io \cep^2 -\frac 12 \io |\nabla \nep|^2 
- \left( \kappa - \frac{C_1^2}{2}\right) \io |\nabla \cep|^2
\\
& \le  \frac \kappa 2\left(\io \nep^2 + \io \cep^2\right) \quad\text{holds in }(0,\min\{T,\tmax\}), 
\end{align*}
which implies the existence of $C_2(T)>0$ such that $\| \nep\cd \|_{L^2(\Omega)}+\| \cep \cd \|_{L^2(\Omega)}\leq C_2(T)$ for all $t\in(0,\min\{T,\tmax\})$. 
\end{proof}

Thanks to these bounds, we can attain the global existence of approximate solutions.  

\begin{lem}\label{Tmax=infty}
For any $\ep\in (0,1)$, we have $\tmax=\infty$. 
\end{lem}
\begin{proof}
Let us assume that for some $\ep\in(0,1)$ we had $\tmax < \infty$. 
Then in light of Lemma \ref{lem;L2;n;ap} and Lemma \ref{lem;regurality;u}, 
$\|\uep\|_{L^\infty(\Omega \times (0,\tmax))}$ were finite. Testing \eqref{apcp2} by $-\Delta \cep$ and combining this bound with, once again, the bound for $\|\nep\|_{L^\infty((0,\tmax);L^2(\Omega))}$ from Lemma \ref{lem;L2;n;ap} would yield 
boundedness of $\|\nabla \cep\|_{L^\infty((0,\tmax);L^2(\Omega))}$, 
which, with $L^p$-$L^q$ estimates for 
the Neumann heat semigroup (see \cite[Lemma 1.3]{win_aggregationvs}) 
could be turned into 
a bound for $\|\nabla \cep\|_{L^\infty((0,\tmax);L^4(\Omega))}$. Another application of the $L^p$-$L^q$ estimates, this time in \eqref{apcp1}, would establish 
the boundedness of $\|\nep\|_{L^\infty(\Omega\times (0,\tmax))}$, 
by the extensibility criterion contradicting $\tmax<\infty$.
\end{proof}

Having achieved global-in-time solutions to the regularized problems, we will next focus on obtaining $\eps$-independent information on our approximate solutions. For small values of $\chi$ a bound on $\int_\Omega \nep^p\cep^{-r}$ for $p>1$ and some suitable $r$ was a key point in obtaining information on $\cep$. This, however, does not work for $p<1$. Nevertheless, in the case of $N=2$ the link between the regularity of $\uep$ and $\nep $ (Lemma \ref{lem;Lr;u}) provides sufficient information on $\uep$ to combine an interpolation inequality with standard semigroup estimates in order to obtain a useful estimate for $\cep$.

\begin{lem}\label{lem;Lresti;cep;N=2}
  Let $N=2$ and let $r\in (1,\infty)$. 
  Then there is $C_r>0$ such that 
  \[
    \lp{r}{\cep \cd}\le C_r \quad \mbox{for all} \ t>0 
    \ \mbox{and all} \ \ep \in (0,1).   
  \]
\end{lem}
\begin{proof}
Let $r \in (1,\infty)$ and let 
$\theta \in (\max\{1,\frac{Nr}{N+r}\}, r)$. 
Then well-known semigroup estimates imply that with some constants $C_1,C_2,C_3>0$ and $\lambda$ being the first positive eigenvalue of $-\Delta$ 
\begin{align}\label{ineq;N=2;Lq;c;semi}\notag 
  \lp{r}{\cep \cd} 
  & \le 
  C_1 \lp{r}{c_0 } 
  + C_2 \int_0^t \left( 1+(t-s)^{-\frac N2(1 - \frac{1}{r})}\right) 
  e^{-\lambda (t-s)}\lp{1}{\nep (\cdot,s)} \,ds  
  \\
  & \quad \, 
  + C_3 \int_0^t \left( 1+(t-s)^{-\frac 12 - \frac N2(\frac{1}{\theta} - \frac{1}{r})}\right) 
  e^{-\lambda (t-s)}\lp{\theta}{(\cep\uep) (\cdot,s)} \,ds
\end{align}
for all $t>0$ and all $\ep\in(0,1)$. 
Here, noting that, in light of multiple applications of the H\"older inequality, for all $\delta \in (0,r-\theta)$, 
\[
  \lp{\theta}{\cep\uep} 
  \le 
  \lp{\frac{\theta(\theta+\delta)}{\delta}}{\uep} \lp{\theta+\delta}{\cep}
  \le 
  \lp{\frac{\theta(\theta+\delta)}{\delta}}{\uep}  
  \lp{1}{\cep}^a\lp{r}{\cep}^{1-a}
\]
for all $t>0$ and all $\ep\in(0,1)$, with $a=\frac{ \frac{1}{\theta+\delta}-\frac 1r  }{ 1-\frac 1r } \in (0,1)$, 
we conclude from \eqref{ineq;N=2;Lq;c;semi} and the bounds on
$\|\nep \|_{L^\infty((0,\infty);L^1(\Omega))}$, $\|\cep\|_{L^\infty((0,\infty);L^1(\Omega))}$ and 
$\|\uep\|_{L^\infty ((0,\infty);L^q (\Omega))}$ with $q\in (1,\infty)$, contained in Lemma \ref{lem;L1;nepcep} and Lemma \ref{lem;Lr;u}, that there is $C_4>0$ such that for all $T>0$  
\[
\lp{r}{\cep\cd}\leq C_4+C_4 \sup_{s\in (0,T)} \lp{r}{\cep(\cdot,s)}^{1-a}
\]
holds for all $t \in (0,T)$ and all $\ep\in(0,1)$, which completes the proof of the lemma with $C_r:=\sup\{x\in\R\mid x\le C_4(1+x^{1-a})\}<\infty$ due to $1-a<1$.  
\end{proof}

\begin{remark}\label{remark:doesntworkforNequal3}
Since $\frac{\theta(\theta+\delta)}{\delta}$ is decreasing with respect to $\delta$ and increasing with respect to $\theta$, inserting $\max\{1,\frac{Nr}{N+r}\}$  
in place of $\theta$ 
and $r-\theta$ instead of $\delta$, 
we see that we need control on $\lp{\varrho}{\uep}$ for some $\varrho>N$ for the reasoning of Lemma \ref{lem;Lresti;cep;N=2} to work. In consequence, this approach is not applicable in the setting of $N=3$, compare Lemma \ref{lem;Lr;u}.
\end{remark}

In the case $N=3$ obtaining a bound similar to the one provided by Lemma \ref{lem;Lresti;cep;N=2} seems to be rather difficult because if we consider $\frac d{dt}\int_\Omega\cep^q$ with $q>1$ we have to treat $\int_\Omega\nep\cep^{q-1}$, which we lack information on. However, considering $\frac d{dt}\int_\Omega\cep^q$ with $q<1$ enables us to obtain a space-time bound on $|\nabla \cep^\frac{q}2|^2$, which in a second step can at least be transformed into space-time information on $\cep$ in the case $N=3$, and in a third step helps to derive a bound on $\int_0^T\int_\Omega|\nabla \cep|^r$ for some $r\in[1,\tfrac54)$ for both of the cases $N=2$ and $N=3$. 

\begin{lem}\label{lem;L2;nablac^q/2}
For all $q\in (0,1)$ and any $T>0$ there exists $C(T)>0$ such that 
\[
  \iio |\nabla \cep^\frac q2|^2 \le C(T) 
\]
holds for all $\ep\in (0,1)$. 
\end{lem}
\begin{proof}
We let $q\in (0,1)$ and $T>0$. 
Multiplying \eqref{apcp2} by $\frac 1q \cep^{q-1}$ and 
integrating over $\Omega$, 
we derive from integration by parts and $\nabla \cdot \uep=0$ in $\Omega\times (0,\infty)$ that for any $\ep\in(0,1)$ 
\begin{align}\label{ineq;testing;cp}
\frac 1q \frac d{dt} \io \cep^q 
&= 
\frac{4 (1-q)}{q^2} \io |\nabla \cep^\frac{q}{2}|^2 
+ \io \nep \cep^{q-1} -\io \cep^q 
\ge 
\frac{4 (1-q)}{q^2} \io |\nabla \cep^\frac{q}{2}|^2 
- \io \cep^q
\end{align}
is valid on $(0,T)$. Upon integration of \eqref{ineq;testing;cp} over $(0,T)$ we infer from Lemma \ref{lem;L1;nepcep} and the positivity of $\cep$, that there exists $C>0$ such that 
\begin{align*}
 \frac{4(1-q)}{q^2} \iio |\nabla \cep^{\frac{q}{2}}|^2 
 \le C\left( T+\frac 1q \right) 
\end{align*}
holds for all $\ep\in(0,1)$. 
\end{proof}
We refine the bound of Lemma \ref{lem;L2;nablac^q/2} to a space-time bound on $\cep$ in the case of $N=3$, the case of $N=2$ already being covered by Lemma \ref{lem;Lresti;cep;N=2}.
\begin{lem}\label{lem;Lresti;cep;N=3}
 Let $N=3$ and let $r\in [1,\frac{5}{3})$. 
 Then for all $T>0$ there is $C_r(T)>0$ such that  
 \[
   \iio \cep^r \le C_r(T) 
 \]
 holds for all $\ep\in (0,1)$. 
\end{lem}
\begin{proof}
Let $T>0$ and $r\in[1,\frac{5}{3})$. We fix $q\in(\frac{1}{3},1)$ such that $r<q+\frac{2}{3}$, which in turn implies 
\begin{align}\label{relation;q;inproof;N=3}
\frac{6(r-1)}{3q-1}<2.
\end{align}
Making use of the Gagliardo--Nirenberg inequality we find $C_1>0$ such that
\begin{align*}
\iio \cep^r=\int_0^T \biglp{\frac{2r}{q}}{\cep^\frac{q}{2}}^\frac{2r}{q}\leq C_1\int_0^T\biglp{2}{\nabla\cep^\frac{q}{2}}^{\frac{2r}{q}a}\biglp{\frac2q}{\cep^\frac{q}{2}}^{\frac{2r}{q}(1-a)}+C_1\int_0^T\biglp{\frac{2}{q}}{\cep^\frac{q}{2}}^\frac{2r}{q}
\end{align*}
is satisfied with $a=\frac{\frac{q}{2}-\frac{q}{2r}}{\frac{q}{2}+\frac{1}{3}-\frac{1}{2}}=\frac{3q(r-1)}{r(3q-1)}$. In light of Lemma \ref{lem;L1;nepcep} we hence obtain $C_2>0$ such that
\begin{align*}
\iio\cep^r\leq C_2\int_0^T\biglp{2}{\nabla\cep^\frac{q}{2}}^\frac{6(r-1)}{3q-1}+C_2T.
\end{align*}
Drawing on \eqref{relation;q;inproof;N=3} and Lemma \ref{lem;L2;nablac^q/2}, a final application of Young's inequality entails the existence of $C_3>0$ satisfying
\begin{align*}
\iio\cep^r\leq (C_3+C_2)T,
\end{align*}
which concludes the proof.
\end{proof}

For both cases $N=2$ and $N=3$ the information on $\cep$ contained in Lemma \ref{lem;Lresti;cep;N=2} and Lemma \ref{lem;Lresti;cep;N=3}, respectively, suffice to ensure an additional spatio-temporal bound on $\nabla \cep$.

\begin{lem}\label{lem;Lr;nablac}
For all $T>0$ and all $r\in [1,\frac 54)$ there is $C(T)>0$ such that 
\[
  \iio |\nabla \cep|^r \le C(T). 
\]
holds for all $\ep\in (0,1)$.  
\end{lem}
\begin{proof}
Let $T>0$ and let $r\in [1,\frac 54)$. 
Since the relation $\frac r{2-r} < \frac 53$ holds because of the inequality $r< \frac{5}{4}$, 
we can find $q\in (0,1)$ such that 
\[
  \frac{r(2-q)}{2-r} < \frac 53. 
\]
Then Lemmata \ref{lem;Lresti;cep;N=2}, \ref{lem;L2;nablac^q/2} and \ref{lem;Lresti;cep;N=3} 
show that there is $C_1(T)>0$ such that 
\begin{align}\label{ineq;cepnaep;nacp}
  \iio \cep^{\frac{r(2-q)}{2-r}} \le C_1(T) 
\quad 
  \mbox{and} 
\quad 
  \iio \cep^{q-2} |\nabla \cep|^2 = \frac4{q^2}\iio |\na \cep^\frac q2|^2 \le C_1(T) 
\end{align}
is valid for all $\ep\in(0,1)$. 
Now, by virtue of the Young inequality with exponents $\f2r$ and $\f2{2-r}$, 
we obtain that for any $\ep\in(0,1)$,
\begin{align}\label{ineq;nacpr} 
  \iio |\na \cep|^r  &= \iio \cep^{\frac{r(q-2)}{2}}|\na \cep|^r \cdot \cep^{\frac{r(2-q)}{2}} 
\le 
  \iio \cep^{q-2} |\na \cep|^2 + \iio \cep^{\frac{r(2-q)}{2-r}}. 
\end{align}
Thus, combination of \eqref{ineq;cepnaep;nacp} and \eqref{ineq;nacpr} 
ensures that this lemma holds. 
\end{proof}

\section{Key relation for existence of generalized solutions}\label{sec4}
The supersolution property of Definition \ref{def;weak(p,q)supersol} is based on the inequality \eqref{ineq;def;weak(p,q)}. Here we show that the approximate solutions satisfy a similar relation, on which we will base the proof of \eqref{ineq;def;weak(p,q)} for the solution we are constructing, but also, prior to that and with $\varphi\equiv 1$, some further a priori estimates. This section closely follows \cite[Section 4]{lankeit_winkler}.
%
%
%
%

\begin{lem}\label{lem;intnpc-r}
Let $T>0$ and $p,q\in (0,1)$. Then 
\begin{align*}
  &- \iio \nep^p \cepq \varphi_t + \io \nepp (\cdot,T)\cepq (\cdot,T) \varphi (\cdot,T) - 
 \io n_0^p c_0^q \varphi (\cdot,0) 
 \\
 &= \iio \frac{4(1-p)q-4q^2 -p\frac{(1-p)^2\chi^2}{(1+\ep\nep)^2}}{pq(\frac{p\chi}{1+\ep\nep}+1-q)}
       \cepq|\nabla \nep^{\frac{p}{2}}|^2\varphi 
\\ 
   &\quad\, + \iio \frac{4}{q}\Big(\frac{p\chi}{1+\ep\nep}+ 1-q \Big) 
                \Big|\nep^{\frac{p}{2}}\nabla \cep^{\frac{q}{2}} 
                - \frac{\frac{(1-p)\chi}{1+\ep\nep}+2q}{2(\frac{p\chi}{1+\ep\nep}+1-q)}
                \cep^{\frac{q}{2}}\nabla \nep^{\frac{p}{2}} \Big|^2\varphi
\\
  &\quad\, + \iio \frac{2((1-p)\ep\nep-p)}{q(1+\ep\nep)^2}
     \nep^\frac{p}{2}\cepq\nabla \nep^{\frac{p}{2}}\cdot \nabla \varphi 
  + \iio \Big( 1-\frac{p\chi}{q(1+\ep\nep)}\Big)\nepp\cepq\Delta \varphi 
\\
  & \quad \, 
  - q \iio \nepp\cepq \varphi + q\iio \nep^{p+1}\cep^{q-1} \varphi 
  + \iio \nepp \cepq \uep\cdot \nabla \varphi
\end{align*}
holds for all $\ep\in(0,1)$ and all $\varphi \in C^\infty(\ol{\Omega}\times [0,T])$ with 
$\frac{\pa \varphi}{\pa \nu}=0$ on $\pa \Omega\times (0,T)$. 
\end{lem}
\begin{proof}
Let $p\in (0,1)$ and $q\in (0,1)$. We then have to compute $\frac d{dt} \io \nepp\cepq\varphi$. These calculations are based on integration by parts and re-ordering of the terms and are straightforward, but rather long. We refer to \cite[Lemma 4.1]{lankeit_winkler} for a detailled proof. The only difference to the fluid-free case treated there is given by the terms 
\begin{align*}
 p \io \nep^{p-1} \cepq \varphi \nabla \nep \cdot \uep 
 + q \io \nepp \cep^{q-1} \varphi \nabla \cep \cdot \uep 
 = \io \varphi \nabla (\nepp\cepq) \cdot  \uep 
 = \iio \nepp \cepq \uep\cdot \nabla \varphi,
\end{align*}
where we have used that $\nabla \cdot \uep =0$.
\end{proof}
The following lemma, which tells us that one of the coefficients in Lemma \ref{lem;intnpc-r} 
becomes positive, 
helps to turn the latter into some $\ep$-independent estimates. 
\begin{lem}\label{lem;1step}
Let $\chi > 0$, $p\in (0,1)$ be such that $p<\frac{1}{\chi^2}$ and let
\begin{align}\label{def;q+-}
  q_{\pm}(p) := \frac{1-p}{2}(1\pm \sqrt{1-p\chi^2})\in(0,1). 
\end{align}
Then for any choice of $q\in (q_-(p),q_+(p))$ there is $C>0$ such that 
\[
  \frac{4(1-p)q-4q^2 -p\frac{(1-p)^2\chi^2}{(1+\ep\nep)^2}}{pq(\frac{p\chi}{1+\ep\nep}+1-q)}
\ge C
\]
holds for all $\ep\in (0,1)$.
\end{lem}
\begin{proof}
The elementary proof can be found in \cite[Lemma 4.2]{lankeit_winkler}. 
\end{proof}
\begin{lem}\label{lem;estifromIneq} 
Let $\chi>0$, $p\in (0,1)$ be such that $p<\frac{1}{\chi^2}$, and let $q\in (q_-(p),q_+(p))$, 
with $q_\pm(p)$ as defined in \eqref{def;q+-}.  
Then for each $T>0$ there exists $C(p,q,T)>0$ fulfiling 
\begin{align}\label{eq:4.6}
  \iio \cepq |\nabla \nep^\frac{p}{2}|^2 \le C(p,q,T),
\end{align}
and 
\begin{align*}
\iio |\nabla \nep^{\frac p2}|^2 \le C(p,q,T)
\end{align*}
as well as 
\begin{align*}
\iio \frac{4}{q}\left(\frac{p\chi}{1+\ep\nep}+ 1-q \right) 
                \left|\nep^{\frac{p}{2}}\nabla \cep^{\frac{q}{2}} 
                - \frac{\frac{(1-p)\chi}{1+\ep\nep}+2q}{2(\frac{p\chi}{1+\ep\nep}+1-q)}
                \cep^{\frac{q}{2}}\nabla \nep^{\frac{p}{2}} \right|^2 \le C(p,q,T)
\end{align*}
and 
\begin{align}\label{eq:np+1cq-1}
\iio \nep^{p+1} \cep^{q-1} \le C(p,q,T) 
\end{align}
for all $\ep\in (0,1)$. 
\end{lem}
\begin{proof}
We assume that $T>0$, let $p\in (0,1)$ satisfy $p<\frac{1}{\chi^2}$, 
and let $q\in (q_-(p),q_+(p))$. 
Then Lemma \ref{lem;intnpc-r} with $\varphi \equiv 1$ 
together with Lemma \ref{lem;1step} 
provides $C_1>0$ such that 
\begin{align}\notag 
 &C_1 \iio \cepq|\nabla \nep^{\frac{p}{2}}|^2  
+ \iio \frac{4}{q}\Big(\frac{p\chi}{1+\ep\nep}+ 1-q \Big) 
                \left|\nep^{\frac{p}{2}}\nabla \cep^{\frac{q}{2}} 
                - \frac{\frac{(1-p)\chi}{1+\ep\nep}+2q}{2(\frac{p\chi}{1+\ep\nep}+1-q)}
                \cep^{\frac{q}{2}}\nabla \nep^{\frac{p}{2}} \right|^2 
\\ \notag  
& \quad \, 
  + q\iio \nep^{p+1}\cep^{q-1} + \io n_0^p c_0^q 
\\ \label{ineq;estifromIneq;step1}
    &\le \io \nepp (\cdot,T)\cepq (\cdot,T)  + q \iio \nepp\cepq 
\end{align}
holds for all $\ep\in (0,1)$. Noting that $\frac q{1-p} < \frac {q_+(p)}{1-p} <1$, 
from Young's inequality we obtain that 
\begin{align}\label{ineq;estifromIneq;step2}
  \io \nep^p\cep^q \le p \io \nep + (1-p) \io \cep^{\frac{q}{1-p}} 
\le p \io \nep + q \io \cep + \frac{1-p-q}{1-p} \quad \text{on } (0,T)
\end{align}
for all $\ep\in(0,1)$. Thus, a combination of \eqref{ineq;estifromIneq;step1} and \eqref{ineq;estifromIneq;step2} with Lemma \ref{lem;L1;nepcep} provides $C_2(p,q,T)>0$ satisfying  
\begin{align}\label{esti;cepnablanep}
  \iio \cepq |\nabla \nep^\frac{p}{2}|^2 \le C_2(p,q,T),
\end{align}
and 
\begin{align*}
\iio \frac{4}{q}\Big(\frac{p\chi}{1+\ep\nep}+ 1-q \Big) 
                \left|\nep^{\frac{p}{2}}\nabla \cep^{\frac{q}{2}} 
                - \frac{\frac{(1-p)\chi}{1+\ep\nep}+2q}{2(\frac{p\chi}{1+\ep\nep}+1-q)}
                \cep^{\frac{q}{2}}\nabla \nep^{\frac{p}{2}} \right|^2 \le C_2(p,q,T),
\end{align*}
as well as  
\begin{align*}
\iio \nep^{p+1} \cep^{q-1} \le C_2(p,q,T)
\end{align*}
for all $\ep\in (0,1)$.  
Moreover, aided by the lower bound for $\cep$ contained in \eqref{ineq;lower;c}, we infer that 
\begin{align*}
\iio |\nabla \nep^{\frac p2}|^2 \le \frac{C_2(p,q,T) e^{qT}}{(\inf c_0)^q}
\end{align*}
is valid for all $\ep \in (0,1)$. 
\end{proof}
\section{Further uniform-in-$\ep$ estimates}
In this section we consider additional uniform-in-$\ep$ estimates, 
which are required to obtain some of the convergence properties listed in Lemma \ref{lem;conv}. In particular, the first of these estimates will be responsible for the regularity of the first solution component and exclusion of persistent Dirac type singularities.
We first recall the following lemma which enables us to pick suitable constants 
$p\in (0,1)$ and $q\in (0,1)$ in the proof of Lemma \ref{lem;Lr(Lr;n)}. 
\begin{lem}\label{lem;infof(1-q)/p}
Let $\chi>0$ and $q_\pm (p)$ be as defined in \eqref{def;q+-}. 
Then 
\[
\inf\Big\{ \frac{1-q}{p} \, \Big|  \, 0<p<\min\big\{ 1,\tfrac 1{\chi^2}\big\} \ \mbox{and} \ q\in \big(q_-(p),q_+(p)\big)  \Big\} 
  = \begin{cases}
    1 & \text{if }\chi \le 1, 
  \\
    \chi & \text{if }\chi\in (1,2), 
  \\
    1+\frac {\chi^2}{4} & \text{if }\chi\geq2. 
  \end{cases}  
\]
\end{lem}
\begin{proof}
A proof is given in \cite[Lemma 5.1]{lankeit_winkler}. 
\end{proof}
%
The following proof for $L^r$-regularity of $\nep$ for some $r>1$ resembles the proof of \cite[Lemma 5.2]{lankeit_winkler}; however, for the $3$-dimensional case, we have to use a quite different source for estimates for $\cep$.
\begin{lem}\label{lem;Lr(Lr;n)}
Assume that $\chi>0$ satisfies \eqref{condi;chi;generalized}. 
Then there exists $r>1$ such that for each $T>0$ 
one can find $C(T)>0$ such that 
\begin{align*}
\iio \nep^r \le C(T) 
\end{align*}
holds for all $\ep\in (0,1)$. 
\end{lem} 
\begin{proof}
In the case $N=2$ we pick arbitrary $p\in(0,\min\{1,\tfrac{1}{\chi^2}\})$ and $q\in\big(q_-(p),q_+(p)\big)$, with $q_\pm(p)$ as defined in \eqref{def;q+-}, and can see from Young's inequality that  
for all $r\in [1,p+1)$ and all $\ep\in(0,1)$ the inequality 
\begin{align}\label{ineq;Young;nr}
  \iio \nep^r =\iio n^r c^{\f{q-1}{p+1}r} c^{\f{1-q}{p+1}r} \le 
    \iio \nep^{p+1}\cep^{q-1} 
    + \iio \cep^{\frac{(1-q)r}{p+1-r}} 
\end{align}
holds. 
Then, we make use of Lemmata \ref{lem;estifromIneq} and \ref{lem;Lresti;cep;N=2} 
to derive from \eqref{ineq;Young;nr} 
that there exists $C_1(T)>0$ such that  
\[
  \iio \nep^r \le C_1(T) 
\]
holds for all $r\in [1,p+1)$ and all $\ep\in(0,1)$. \\

On the other hand, in the case $N=3$ 
we need to choose $p\in (1,\min\{1,\frac{1}{\chi^2}\})$, 
$q \in (q_-(p),q_+(p))$ such that 
for each $T>0$ there is $C_2(T)>0$ satisfying 
\begin{align}\label{ineq;L1-qr..;nr}
  \iio \cep^{\frac{(1-q)r}{p+1-r}} \le C_2(T),
\end{align}
for all $\ep\in(0,1)$ with some $r>1$. Due to $\chi<\frac 53$, we can rely on Lemma \ref{lem;infof(1-q)/p}
to find $p\in (1,\min\{1,\frac{1}{\chi^2}\})$ and $q \in (q_-(p),q_+(p))$ 
such that $\frac{1-q}{p}=\max\{1,\chi\}<\frac{5}{3}$. 
Hence, choosing $r>1$ sufficiently close to $1$ we can ensure that 
\[
  \frac{(1-q)r}{p+1-r} <\frac{5}{3}. 
\]
Plugging this into Lemma \ref{lem;Lresti;cep;N=3} yields \eqref{ineq;L1-qr..;nr}, 
showing that this lemma also holds in the case $N=3$. 
\end{proof}
We next prepare a spatio-temporal bound on the mixed quantitiy $(\nep^p\cep^q)^\gamma$ for some $\gamma>\frac32$, which will be a key ingredient in verifying the supersolution property for the limit functions to be obtained in Section \ref{sec7}.
\begin{lem}\label{lem;esti;npcq}
Assume $\chi>0$ to satisfy \eqref{condi;chi;generalized} 
and let $p\in (0,\min\{1,\frac 1{\chi^2}\})$ 
and $q\in (q_-(p),q_+(p))$, with $q_\pm(p)$ as defined in \eqref{def;q+-}, fulfil
\begin{align}\label{condi;adi;pq}
p+ \frac{3q}{5} < \frac{2}{3}.
\end{align}
Then there is $\gamma > \frac{3}{2}$ such that for all $T>0$ there is $C(T)>0$ such that 
\[
  \iio (\nepp\cepq)^\gamma \le C(T)
\] 
holds for all $\ep\in (0,1)$. 
\end{lem}
\begin{proof}
Let $p\in (0,\min\{1,\frac 1{\chi^2}\})$ 
and $q\in (q_-(p),q_+(p))$ be numbers satisfying \eqref{condi;adi;pq}, and 
let $r>1$ be a constant as obtained in Lemma \ref{lem;Lr(Lr;n)}. 
Then, since the relation 
\begin{align}\label{lera;pq;gamma;npcp}
  \frac pr + \frac{3q}{5} < p+\frac{3q}5 < \frac{2}{3} 
\end{align}
holds, 
we can find constants $\gamma_1,\gamma_2>1$ such that 
\begin{align}\label{rela;gamma1gamma2}
  \gamma_1 < \frac{r}{p} 
\quad 
  \mbox{and} 
\quad 
  \gamma_2 < \frac{5}{3q}
\end{align} 
as well as 
\begin{align}\label{rela;gamma1gamma223}
  \frac{1}{\gamma_1} + \frac{1}{\gamma_2} < \frac{2}{3} 
\end{align}
are valid. Now, we put $\gamma := \frac{\gamma_1\gamma_2}{\gamma_1+\gamma_2}$ and obtain from \eqref{rela;gamma1gamma223} that $ \gamma > \frac 32$. 
Moreover, combination of \eqref{rela;gamma1gamma2} 
and Lemmata \ref{lem;Lr(Lr;n)}, \ref{lem;Lresti;cep;N=2} and \ref{lem;Lresti;cep;N=3} 
with Young's inequality provides $C_1(T)>0$ such that 
\[
  \iio (\nep^p\cep^q)^\gamma 
  \le \iio \nep^{p\gamma_1} + \iio \cep^{q\gamma_2} 
  \le C_1(T)
\]
holds for all $\ep\in(0,1)$.
\end{proof}
\begin{remark}
By continuity we can find some small $p\in (0,\min\{1,\frac{1}{\chi^2}\})$ and $q\in (q_-(p),q_+(p))$ such that 
\eqref{condi;adi;pq} holds, because 
\eqref{condi;adi;pq} is satisfied for $p=0$ and $q=0=q_-(0)$. 
\end{remark}
\section{Time regularity}\label{sec6}
This section provides regularity information on the time derivatives of the approximate solutions, or suitable transformations thereof,
which is required for the application of an Aubin--Lions type lemma. 
The proofs of the following three lemmata are based on duality arguments. 
\begin{lem}\label{lem;time;n}
Assume that $\chi>0$ satisfies \eqref{condi;chi;generalized} and 
$p\in (0,\min\{1,\frac 1{\chi^2}\})$. 
Then for all $T>0$ there exists $C(T)>0$ such that 
\[
  \int_0^T \|\pa_t (\nep+1)^\frac{p}{2}\|_{(W^{1,\infty}_0(\Omega))^\ast} \le C(T)
\]
for all $\ep\in (0,1)$. 
\end{lem}
\begin{proof}
Let $\psi \in C_0^\infty (\Omega)$ be such that $\| \psi \|_{W^{1,\infty}(\Omega)} \le 1$. 
Noting from the Young inequality that 
\[
  \left|\frac p2 \io (\nep+1)^{\frac{p-2}{2}}\na \nep \cdot \uep \psi \right| 
  \le \frac{p}{4}\lp{\infty}{\psi} 
  \left(  
    \io (\nep+1)^{p-2}|\na \nep|^2 + \io |\uep|^2
  \right), 
\]
from arguments similar to those in \cite[Lemma 7.1]{lankeit_winkler} we obtain that 
\[
 \left| \io \pa_t (\nep+1)^{\frac p2} \psi \right| 
 \le C_1 \left\{ \io |\na \nep^{\frac p2}|^2 + \io |\na \cep^{\frac 13}| +\io |\uep|^2 + 1 \right\}
\]
on $(0,T)$ for all $\ep\in (0,1)$ with some $C_1>0$. Thus, the inequality 
\begin{align*} 
   \|\pa_t (\nep+1)^\frac{p}{2}\|_{(W^{1,\infty}_0(\Omega))^\ast} 
  & =  \sup \Bigg\{\, \Big| \io \pa_t (\nep+1)^{\frac p2} \psi \Big| \, \Bigg| \, 
      \psi\in C^\infty_0 (\Omega), \ \| \psi \|_{W^{1,\infty}(\Omega)}\le 1 \, \Bigg\}  
\\
  & \le C_1 \Big\{ \io |\na \nep^{\frac p2}|^2 + \io |\na \cep^{\frac 13}| +\io |\uep|^2 + 1 \Big\},
  \end{align*}
together with Lemmata \ref{lem;estifromIneq} and \ref{lem;L2;nablac^q/2} and Corollary \ref{lem;Lr;u}
implies that there is some $C_2(T)>0$ such that 
\[  
  \int_0^T \|\pa_t (\nep+1)^\frac{p}{2}\|_{(W^{1,\infty}_0(\Omega))^\ast} \le C_2(T)
\]
holds for all $\ep\in (0,1)$. 
\end{proof}
\begin{lem}\label{lem;time;c}
Let $\chi >0$ and $T>0$.  
Then there exists $C(T)>0$ such that 
\[
  \int_0^T \|\cept\|_{(W^{1,\infty}_0(\Omega))^\ast} \le C(T)
\]
for all $\ep\in (0,1)$. 
\end{lem}
\begin{proof}
Let $\psi \in C^\infty_0 (\Omega)$ be such that $\| \psi \|_{W^{1,\infty} (\Omega)} \le 1$. 
Since 
the Young inequality and 
arguments similar to those in the proof of \cite[Lemma 7.2]{lankeit_winkler} imply 
\[
  \|\cept\|_{(W^{1,\infty}_0(\Omega))^\ast} 
  \le \io |\na \cep| + \io \cep + \io \nep 
  + \io \cep^{\frac{30}{19}} 
  + \io |\uep|^{\frac{30}{11}}, 
\]
from 
Corollary \ref{lem;Lr;u}, and Lemmata \ref{lem;L1;nepcep}, \ref{lem;Lresti;cep;N=2} 
and \ref{lem;Lresti;cep;N=3} we deduce 
that there exists some $C_1(T)>0$ such that 
\[
  \int_0^T \|\cept\|_{(W^{1,\infty}_0(\Omega))^\ast} \le C_1(T) 
\]
holds for all $\ep\in(0,1)$. 
\end{proof}
%
\begin{lem}\label{lem;time;u}
Let $\chi>0$. 
Then there exists $C>0$ such that 
\[
  \|\uept (\cdot,t)\|_{(W^{1,4}_{0,\sigma}(\Omega))^\ast} \le C
\]
holds for all $t\in (0,\infty)$ and  $\ep\in (0,1)$. 
\end{lem}
\begin{proof}
Let $\psi \in W^{1,4}_{0,\sigma}(\Omega;\R^N)$ be such that $\|\psi\|_{W^{1,4}(\Omega)} \le 1$. 
An application of the Young inequality entails that 
\begin{align*}
  \left| \io \uept \psi \right| 
  &= \left| - \io \nabla \uep \cdot \nabla \psi + \io \nep\nabla \Phi\cdot \psi  \right|
\\
  &\le \lp{\frac{4}{3}}{\na \uep}\lp{4}{\na \psi} + \lp{1}{\nep} \lp{\infty}{\nabla \Phi} \lp{\infty}{\psi}
\end{align*}
holds for all $\ep \in (0,1)$. 
Thus, from Lemmata \ref{lem;L1;nepcep} and \ref{lem;Lr;u} we establish that with some $C_1>0$,
\[  
  \|\uept (\cdot,t)\|_{(W^{1,4}_{0,\sigma}(\Omega))^\ast} \le C_1  
\]
for all $t\in (0,\infty)$ and all $\ep\in (0,1)$, completing the proof of this lemma.
\end{proof}
\section{Convergence: Proof of Theorem \ref{mainthm}}\label{sec7}
In this section we complete the proof of Theorem \ref{mainthm}. 
The following lemma is a consequence of 
the estimates prepared in Sections \ref{sec4}--\ref{sec6}. 
\begin{lem}\label{lem;conv}
  Suppose that $\chi>0$ satisfies \eqref{condi;chi;generalized} and 
  let $p\in (0,1)$ and $q\in (0,1)$ be such that $p < \frac{1}{\chi^2}$ 
  and $q \in (q_-(p),q_+(p))$, where $q_\pm(p)$ are defined as in \eqref{def;q+-}, and assume they satisfy \eqref{condi;adi;pq}  as well. 
  Then there exist $(\ep_j)_{j\in \mathbb{N}}\subset (0,1)$ and functions $n,c,u$ such that 
  $\ep_j \searrow 0$ as $j\to \infty$ and 
\begin{align} \label{conv;ninL1}
  &\nep \to n  
  && \mbox{in} \ L^1_{loc}(\ol{\Omega}\times [0,\infty)) \ 
  \mbox{and a.e.\ in}\ \Omega\times (0,\infty), \\
\label{conv;ninmorethanL1}
  &\nep \to n  
  && \mbox{in} \ L^s_{loc}(\ol{\Omega}\times [0,\infty)) \mbox{ for some } s>1, 
\\  \label{conv;nanepp/2}  
  &\na \nep^{\frac p2} \rightharpoonup \na n^{\frac p2} 
  && \mbox{in} \ L^2_{loc} (\ol{\Omega}\times [0,\infty);\R^N), 
\\ \label{conv;c}
 &\cep \to c  
  && \mbox{in} \ L^r_{loc}(\ol{\Omega}\times [0,\infty)) \ 
  \mbox{and a.e.\ in}\ \Omega\times (0,\infty) \ 
  \mbox{{\rm (}for all} \ r\in [1,\tfrac53)\mbox{\rm )}, 
\\\label{conv;nac}
  &\na \cep \rightharpoonup \na c 
  && \mbox{in} \ L^1_{loc} (\ol{\Omega}\times [0,\infty);\R^N), 
\\\label{conv;nacq2}
  &\na \cep^{\frac q2} \rightharpoonup \na c^{\frac q2} 
  && \mbox{in} \ L^2_{loc} (\ol{\Omega}\times [0,\infty);\R^N), 
\\\label{conv;u}
  &\uep \to u  
  && \mbox{in} \ L^s_{loc}([0,\infty);L^s_\sigma(\Omega))\ 
  \mbox{and a.e.\ in}\ \Omega\times (0,\infty) \ 
  \mbox{{\rm (}for all} \ s\in [1,3)\mbox{\rm )}, 
\\ \label{conv;nau}
  &\na \uep \rightharpoonup \na u 
  && \mbox{in} \ L^1_{loc} (\ol{\Omega}\times [0,\infty);\R^{N\times N})  
\end{align} 
as $\ep=\ep_j \searrow 0$, and 
\begin{align}\label{massconv;n}
  \io n\cd = \io n_0 \quad \mbox{for a.e.} \ t>0. 
\end{align}
Moreover, 
\[
 n\ge 0, \;\; c\ge 0\qquad \text{ a.e. in } \Om\times(0,\infty).
\]
\end{lem}
\begin{proof}
Let $p,q\in (0,1)$ be such that $p<\frac 1{\chi^2}$ and $q\in (q_-(p),q_+(p))$.  
Then Lemmata \ref{lem;L1;nepcep}, \ref{lem;estifromIneq} and \ref{lem;time;n} enable us to see that 
\[
  \left( (\nep+ 1)^\frac{p}{2} \right)_{\ep \in (0,1)} \quad \mbox{is bounded in} \ 
  L^2_{loc} ([0,\infty);W^{1,2}(\Omega)) 
\]
and 
\[
    \left( \pa_t (\nep+ 1)^\frac{p}{2} \right)_{\ep \in (0,1)} \quad \mbox{is bounded in} \ 
  L^2_{loc} ([0,\infty);(W^{1,\infty}_0(\Omega))^\ast), 
\]
which together with the Aubin--Lions lemma \cite[Corollary 8.4]{Simon} 
provides $(\ep_j)_{j\in \mathbb{N}}$ satisfying $\ep_j\searrow 0$ as $j\to \infty$ 
and a function $(n+1)^\frac{p}{2}:=v \in L^2_{loc}(\ol{\Omega} \times [0,\infty))$ such that 
$(\nep+1)^\frac{p}{2} \to v$ in $L^2_{loc} (\ol{\Omega}\times [0,\infty))$ 
and a.e.\ in $\Omega\times (0,\infty)$ as $\ep=\ep_j \searrow 0$. 
Furthermore, we deduce from Lemma \ref{lem;estifromIneq} that 
\eqref{conv;nanepp/2} holds. 
Moreover, aided by Lemma \ref{lem;Lr(Lr;n)} 
we invoke the Vitali convergence theorem to verify that 
\eqref{conv;ninL1}, even \eqref{conv;ninmorethanL1}, and, due to Lemma \ref{lem;L1;nepcep}, \eqref{massconv;n} hold. 
We next note that for $a\in(1,\f54)$ 
%
%
from Lemmata \ref{lem;Lr;nablac} and \ref{lem;time;c} we can infer that 
\[
  (\cep)_{\ep \in (0,1)} \quad \mbox{is bounded in} \ L^{a}_{loc} ([0,\infty); W^{1,a}(\Omega))  
\]
and 
\[
  (\cept)_{\ep\in (0,1)} \quad 
  \mbox{is bounded in} \ L^1_{loc}([0,\infty); (W^{1,\infty}_0(\Omega))^\ast), 
\]
so that another application of the Aubin--Lions lemma demonstrates the existence of a further subsequence (again denoted by $(\ep_j)_{j\in \mathbb{N}}$) and a function $c\in L^a_{loc}(\Ombar\times[0,\infty))$ such that 
$c_{\ep}\to c$ in $L^a_{loc}(\Ombar\times[0,\infty))$ and a.e.\ in $\Omega\times (0,\infty)$ as $\ep = \ep_j \searrow 0$. 
According to Lemma \ref{lem;Lresti;cep;N=3} (or Lemma \ref{lem;Lresti;cep;N=2}) and, again, Vitali's convergence theorem, this can be improved to \eqref{conv;c}. Lemmata \ref{lem;Lr;nablac} and \ref{lem;L2;nablac^q/2} yield \eqref{conv;nac} and \eqref{conv;nacq2}, respectively.
For $b\in(1,\f32)$, 
we obtain from Corollary \ref{lem;Lr;u} and Lemma \ref{lem;time;u} that  
\[
  (\uep)_{\ep \in (0,1)} \ \mbox{is bounded in} \ 
  L^b_{loc} ([0,\infty); W^{1,b}_{0,\sigma}(\Omega)) 
\]
and 
\[
  (\uept)_{\ep \in (0,1)} \ \mbox{is bounded in} \ 
  L^1_{loc}([0,\infty); (W^{1,4}_{0,\sigma}(\Omega))^\ast)
\]
and similarly invoking the Aubin--Lions lemma and Vitali's theorem together with the second part of Corollary \ref{lem;Lr;u}, we obtain \eqref{conv;u}, whereas \eqref{conv;nau} again is immediate from the first part of Corollary \ref{lem;Lr;u}.
Nonnegativity of $n$ and $c$ follows directly from \eqref{conv;ninL1} and \eqref{conv;c} and nonnegativity of $\nep$ and $\cep$.
\end{proof}
Now we can verify that $(n,c,u)$ is a global weak solution of \eqref{cp2}--\eqref{cp5}. 
\begin{lem}\label{lem;weaksol}
  If $\chi>0$ satisfies \eqref{condi;chi;generalized}, then the triplet $(n,c,u)$ 
  obtained in Lemma \ref{lem;conv} is a global weak solution of 
  \eqref{cp2}--\eqref{cp5}. 
\end{lem}
\begin{proof}
We first note from Lemma \ref{lem;conv} that 
the required regularity  conditions of the solution $(n,c,u)$ are satisfied 
with $\eta \in (0,\f16)$. 
Now, we let $\varphi \in C^\infty_0 (\ol{\Omega}\times [0,\infty))$ 
and $\psi \in C^\infty_{0,\sigma} (\ol{\Omega}\times [0,\infty))$. 
Testing \eqref{apcp2} and \eqref{apcp3} by 
$\varphi$ and $\psi$, respectively, and using integration by parts, 
we derive that 
\[
  -\iiio \cep \varphi_t -\io c_0 \varphi(\cdot,0) = -\iiio \nabla \cep \cdot \nabla \varphi - 
  \iiio \cep \varphi + \iiio \nep \varphi + \iiio \cep\uep \cdot \nabla \varphi    
\]
and 
\[
  -\iiio \uep\cdot \psi_t -\io u_0\cdot \psi (\cdot,0) = 
  -\iiio \na \uep \cdot\nabla \psi - \iiio \nep \na \Phi \cdot \psi
\]
hold for all $\ep \in (0,1)$. 
Then passing to the limit in the above identities as $\ep=\ep_j \searrow 0$ 
on the basis of \eqref{conv;c}, \eqref{conv;nac}, \eqref{conv;ninL1}, \eqref{conv;u} and \eqref{conv;nau} 
leads to 
this lemma. 
\end{proof}
In order to verify that $(n,c,u)$ is also a global weak $(p,q)$-supersolution we also have to obtain the positivity properties present in Definition \ref{def;weak(p,q)supersol}. As a preparatory step we state the following two lemmata, which have already been shown in \cite{lankeit_winkler}. 
\begin{lem}\label{lem;DI}
Let $a,b,T>0$ and 
let $y:(0,T)\to \mathbb{R}$ be a continuously differentiable function 
satisfying 
\[
  y'(t)  \le  - a y^2(t) + b
\quad \mbox{for all} \ t\in (0,T) \ \mbox{at which} \ y(t)>0. 
\]
Then 
\[
  y(t) \le \sqrt{\frac{b}{a}} \coth (\sqrt{ab} t) \quad t \in (0,T). 
\]
\end{lem}
\begin{proof}
This lemma can be found in \cite[Lemma 8.3]{lankeit_winkler}.  
\end{proof}
\begin{lem}\label{lem;oneofkey}
Let $\eta > 0$. 
Then there exists $C>0$ such that every function $\varphi\in C^1(\ol{\Omega})$ fulfiling 
\[
  \big| \{ x\in \Omega \, | \, \varphi(x) > \delta \} \big| > \eta
\]
for some $\delta > 0$ satisfies 
\[
  \io \frac{|\nabla \varphi|^2}{\varphi} \ge C \left \{ \io \ln \frac \delta\varphi  \right\}^2 
  \quad \mbox{or} \quad \io \ln \frac \delta\varphi <0. 
\]
\end{lem}
\begin{proof}
This lemma can be found in \cite[Lemma 8.4]{lankeit_winkler}. 
\end{proof}
Thanks to these lemmata, we can establish the following cornerstone for the proof of the positivity of the functions $n$ and $c$ obtained 
in Lemma \ref{lem;conv}. 
Because of the 
additional presence of the convection term $\uep \cdot \na \cep$,  
modification of the proof of \cite[Lemma 8.5]{lankeit_winkler} is necessary. 
\begin{lem}\label{lem;lower;logn}
  There exists $T>0$ such that for all $t\in (0,T)$, 
  \[
    \inf_{\ep\in (0,1)}\io \ln \nep \cd > -\infty.    
  \]
\end{lem}
%
\begin{proof}
We let $\theta >N$ and put 
\[
  M_\ep (t) := \sup_{\tau \in [0,t]} (\lp{\infty}{\nep(\cdot,\tau)}+ \lp{\theta}{\na \cep (\cdot,\tau)})
\]
for $t\in (0,\infty)$ and $\eps\in(0,1)$. Then $\norm[L^\infty(\Om\times(0,t))]{\nep}\le M_\ep(t)$ for any $t>0$ and $\eps\in(0,1)$, and thanks to 
Lemma \ref{lem;regurality;u},  
we thereby have established that 
\begin{equation}\label{estlinftyu}
   \lp{\infty}{\uep\cd} \le C_1 (1+ M_\ep (t))
\end{equation}
for all $t\in (0,\infty)$ and all $\ep\in (0,1)$. 
From the $L^p$-$L^q$ estimates for the Neumann heat semigroup 
we obtain $C_2\ge 1$ such that with some $C_3>0$,
\begin{align}\label{est:nacep}
  \lp{\theta}{\na \cep \cd} 
& \le 
  \lp{\theta}{\na e^{t(\Delta -1)} c_0 } 
    + \int_0^t\lp{\theta}{\na e^{(t-s)(\Delta -1 )} (\nep(\cdot,s) - (\uep \cdot \na\cep)(\cdot,s)) } \, ds \nn
\\
& \le C_2 \lp{\theta}{\na c_0} 
  + C_2\left( 
      |\Omega|^{\frac{1}{\theta}} M_\ep(t) 
      + C_1M_\ep(t) (1+M_\ep(t)) \right) \int_0^t (1+(t-s)^{-\frac{1}2})\, ds \nn
\\ 
& \le C_2\lp{\theta}{\na c_0} + C_3 \left( M_\ep(t)+ M^2_\ep(t) \right) \left( t+t^\frac{1}{2} \right)
\end{align}
holds for all $t\in(0,\infty)$. Apart from $\norm[\Lom\infty]{\nep(\cdot,t)}\le M_{\ep}(t)$, the definition of $M_\ep(t)$ also ensures that $\norm[\Lom\theta]{\na\cep(\cdot,t)}\le M_{\ep}(t)$ for all $t>0$, $\ep\in(0,1)$, which we also have just used in \eqref{est:nacep}, so that \eqref{ineq;lower;c} implies the existence of $C_4>0$ such that 
\[
 \norm[\Lom{\theta}]{\left(\f{\nep}{(1+\ep\nep)\cep} \na \cep\right)(\cdot,t)} \le C_4M^2_\ep(t)e^t \qquad \text{for } t>0,\; \ep\in(0,1),
\]
and \eqref{estlinftyu} guarantees 
\[
 \norm[\Lom{\theta}]{(\nep\uep)(\cdot,t)} \le 
 C_1 
 |\Omega|^{\frac{1}{\theta}}  
 (M_\ep(t)+M_\ep^2(t)) \qquad \text{for } t>0,\; \ep\in(0,1).
\]
Hence from a similar application of $L^p$-$L^q$ estimates in 
\[
  \lp{\infty}{\nep \cd} 
 \le 
  \lp{\infty}{e^{t\Delta} n_0} + 
  \int_0^t \left\|e^{(t-s)\Delta}\na \cdot \left(
     \frac{\nep}{(1+\ep \nep)\cep}\na \cep - \nep\uep 
  \right)(\cdot,s)\right\|_{L^\infty(\Omega)}
\]
for $t>0$ and $\ep\in(0,1)$, 
we can see that with some $C_5>0$ 
\begin{align}\label{est:nep}
  &\lp{\infty}{\nep \cd} \le\lp{\infty}{n_0}+C_5 (M_\ep(t)+M_\ep^2(t))e^t(t+t^{\f12-\f N{2\theta}}).
%
\end{align}
Adding \eqref{est:nacep} and \eqref{est:nep}, we conclude that with some $C_6>0$, the estimate 
\[
   M_\ep (t) \le 
  \lp{\infty}{n_0} + C_2 \lp{\theta}{\na c_0} + 
  C_6 e^t 
  \left( t + t^{\frac 12 - \frac   N{2\theta}} \right)
  (M_\ep(t) + M^{2}_\ep(t)) 
\]
holds true for all $t>0$ and $\ep\in(0,1)$. 
%
%
%
Now, letting 
\[ 
  T_\ep := \sup\{ t>0 \, | \, M_\ep (t) \le M := \lp{\infty}{n_0} + C_2 \lp{\theta}{\na c_0} + 1\}  
\]
and 
\[
  \tilde{T} := \min\big\{1, 
  4eC_6(M+M^{2})^{-\frac{2\theta}{\theta-N}},  
  T_\ep\big\} \le T_\ep, 
\]
we find that for all $t\in (0,\tilde{T}$) we have 
\begin{align*}
  C_6 e^t
  \left( t + t^{\frac 12 - \frac   N{2\theta}} \right) 
  (M_\ep(t) + M^{2}_\ep(t)) \le C_6 e 
  \left( 2\tilde{T}^{\frac 12 - \frac   N{2\theta}} \right)
  (M + M^{2}) 
  \le \frac 12.  
\end{align*}
Thus, we can see that 
\begin{align*}
  M_\ep (t) \le \lp{\infty}{n_0} + C_2 \lp{\theta}{\na c_0} + \frac{1}{2} < M
  \quad \mbox{for all} \ t\in (0,\tilde{T}), 
\end{align*}
which means that 
\begin{align*}
  T_\ep > \tilde{T} = 
  \min\{1, 
  4eC_6(M+M^{2})^{-\frac{2\theta}{\theta-N}}
  \} =: T. 
\end{align*} 
In conclusion, we infer that for all $t\in (0,T)$ and all $\ep \in (0,1)$,  
\begin{align}\label{shorttimebounds}
  \lp{\infty}{\nep \cd} \le M  \quad \mbox{and}\quad \lp{\theta}{\na \cep \cd} \le M
\end{align}
hold. Now we can follow the proof of \cite[Lemma 8.5]{lankeit_winkler}: The first part of \eqref{shorttimebounds} together with mass conservation of $\nep$ ensures applicability of Lemma \ref{lem;oneofkey} with some positive $\delta$ and $\eta$; and Lemma \ref{lem;oneofkey} and the second part of \eqref{shorttimebounds} hence show that with some $C_7, C_8>0$, 
\[
 \f{d}{dt} \io \ln \f{\delta}{\nep\cd} \le - C_7 \left(\io \ln \f{\delta}{\nep\cd}\right)^2 + C_8  
\]
for every $t\in(0,T)$ satisfying $\io \ln \f{\delta}{\nep\cd}>0$, so that Lemma \ref{lem;DI} proves the claim. 
%
%
\end{proof}
Now we can attain the positivity of $n$ and $c$ which is required in the definition of 
global weak solutions. 
\begin{lem}\label{lem;positivity}
The functions $n$ and $c$ from Lemma \ref{lem;conv} satisfy $n>0$ and $c>0$ a.e.\ in $\Omega\times (0,\infty),$ 
as well as $n^p c^q >0$ a.e.\ in $\pa \Omega \times (0,\infty)$.  
\end{lem}
\begin{proof}
The positivity requirement on $c$ is obviously satisfied due to \eqref{ineq;lower;c} and \eqref{conv;c}. As in \cite[proof of Lemma 8.6]{lankeit_winkler}, we can derive a differential inequality of the form \[\f{d}{dt}\left[-\io \ln \nep - \chi^2\io \ln \cep\right] +\f12 \io |\na \ln \nep|^2 \le C.\] If we use \eqref{ineq;lower;c} and Lemma \ref{lem;lower;logn}, upon integration with respect to time the dissipative term
 yields an $\ep$-independent bound on $\int_{\tau_0}^t\io |\na \ln\nep|^2$ for arbitrary $t>\tau_0\in(0,T)$ with $T$ as in Lemma \ref{lem;lower;logn}, so that finally -- following the reasoning of \cite[proof of Lemma 8.6]{lankeit_winkler} -- we can conclude that $\ln n$ belongs to $L^2_{loc}((0,\infty);W^{1,2}(\Om))$ and to $L^2_{loc}(\partial \Om\times(0,\infty))$ and hence has to be finite almost everywhere.
\end{proof} 
Several terms 
in the relation obtained in Lemma \ref{lem;intnpc-r} 
contain coefficients that should become constants in the limit, but for positive $\ep$ seem much more involved, due to their dependence on $\nep$. In order to verify their convergence, 
we recall the following lemma. 
\begin{lem}\label{lem;conv-1}
Let $(f_\ep)_{\ep\in (0,1)} \subset C^0 ([0,\infty))\cap L^\infty((0,\infty))$  be such that 
\[
  \sup_{\ep\in (0,1)} \norm[L^\infty((0,\infty))]{f_\ep} < \infty 
\]
and that there exists $f \in C^0 ([0,\infty))$ such that 
\[
  f_\ep \to f \quad \mbox{in} \ L^\infty_{loc} ([0,\infty)) \ \mbox{as} \ \ep \searrow 0. 
\]
Then given $\chi>0$ satisfying \eqref{condi;chi;generalized} and taking $n,c,u$ and $(\ep_j)_{j\in \mathbb{N}}$ from Lemma \ref{lem;conv}, 
for all $p\in (0,1)$ and $q\in(0,1)$ such that $p<\frac{1}{\chi^2}$ and $q\in (q_-(p),q_+(p))$, with $q_\pm(p)$ as defined in \eqref{def;q+-}, and each $T>0$ 
we have 
\begin{align*}
  f_\ep(\nep)\cep^{\frac q2}\nabla \nep^{\frac p2} \rightharpoonup f(n) c^\frac{q}{2} \nabla n^{\frac{p}{2}} \quad \mbox{in} \ L^2(\Omega\times (0,T)) \ \mbox{as} \ \ep=\ep_j \searrow 0.  
\end{align*}
\end{lem}
\begin{proof}
A proof, which for obtaining a convergent sequence essentially relies on \eqref{eq:4.6}, and, inter alia, on \eqref{conv;ninL1}, \eqref{conv;nanepp/2}, \eqref{conv;c} for identification of its limit, can be found in \cite[Lemma 8.7]{lankeit_winkler}. 
\end{proof}
Now all tools for the proof of Theorem \ref{mainthm} are 
provided. 
Finally, we give the following lemma which shows that 
the remaining requirements of the definition of global generalized solutions are satisfied. 
\begin{lem}\label{lem;existence}
Suppose that $\chi>0$ satisfies \eqref{condi;chi;generalized}, and let $p\in (0,1)$ and $q\in(0,1)$ be such that 
$p<\frac{1}{\chi^2}$ and $q\in (q_-(p),q_+(p))$, with $q_\pm(p)$ as defined in \eqref{def;q+-}, and assume they satisfy \eqref{condi;adi;pq} as well. 
Then the functions $n,c,u$ constructed in Lemma \ref{lem;conv} 
form a global weak $(p,q)$-supersolution of 
\eqref{cp1} and \eqref{cp4}--\eqref{cp5}  
in the framework of Definition \ref{def;weak(p,q)supersol}. 
\end{lem}
\begin{proof}
In light of Lemma \ref{lem;esti;npcq}, we can find a further subsequence $(\ep_j)_{j\in \mathbb{N}}$ of the sequence 
 obtained in Lemma \ref{lem;conv} and 
$w\in L^{\frac 32 + \eta}_{loc}(\ol{\Omega}\times [0,\infty))$ such that 
\(
  \nepp\cepq \rightharpoonup w \quad 
  \mbox{in} \ L^{\frac 32 + \eta} (\ol{\Omega}\times [0,\infty))  \)
with some $\eta >0$, as $\eps=\eps_j\searrow 0$ and then, relying on \eqref{conv;ninL1} and \eqref{conv;c}, make sure that $w$ coincides with $n^p c^q$, so that 
\begin{equation}\label{conv;npcq}
  \nepp\cepq \rightharpoonup n^pc^q \quad 
  \mbox{in} \ L^{\frac 32 + \eta} (\ol{\Omega}\times [0,\infty)).
\end{equation}
This, together with the fact that $\f{3+2\eta}{1+2\eta}<3$ and \eqref{conv;u}, and with a combination of pointwise convergence in \eqref{conv;ninL1} and \eqref{conv;c} with \eqref{eq:np+1cq-1} asserts the regularity requirements of \eqref{def;supersol;regularity1}. The remainder of the proof follows that of \cite[Lemma 8.8]{lankeit_winkler} closely and we hence restrict ourselves to a rough outline: In order to verify \eqref{ineq;def;weak(p,q)}, we pick a nonnegative function $\varphi\in C_0^\infty(\Ombar\times[0,\infty))$ with $\f{\partial \varphi}{\partial \nu}=0$ on $\partial \Om\times(0,\infty)$ and let $T>0$ be such that $\varphi\equiv 0$ on $\Om\times [T,\infty)$. With this function, Lemma \ref{lem;intnpc-r} turns into 
\begin{align*}  
- \iio &\nep^p \cepq \varphi_t  -  \io n_0^p c_0^q \varphi (\cdot,0) 
 = \iio \frac{4(1-p)q-4q^2 -p\frac{(1-p)^2\chi^2}{(1+\ep\nep)^2}}{pq(\frac{p\chi}{1+\ep\nep}+1-q)}
       \cepq|\nabla \nep^{\frac{p}{2}}|^2\varphi 
\\ 
   &\quad\, + \iio \frac{4}{q}\Big(\frac{p\chi}{1+\ep\nep}+ 1-q \Big) 
                \Big|\nep^{\frac{p}{2}}\nabla \cep^{\frac{q}{2}} 
                - \frac{\frac{(1-p)\chi}{1+\ep\nep}+2q}{2(\frac{p\chi}{1+\ep\nep}+1-q)}
                \cep^{\frac{q}{2}}\nabla \nep^{\frac{p}{2}} \Big|^2\varphi
\\
  &\quad\, + \iio \frac{2((1-p)\ep\nep-p)}{q(1+\ep\nep)^2}
     \nep^\frac{p}{2}\cepq\nabla \nep^{\frac{p}{2}}\cdot \nabla \varphi 
  + \iio \Big( 1-\frac{p\chi}{q(1+\ep\nep)}\Big)\nepp\cepq\Delta \varphi 
\\
  & \quad \, 
  - q \iio \nepp\cepq \varphi + q\iio \nep^{p+1}\cep^{q-1} \varphi 
  + \iio \nepp \cepq \uep\cdot \nabla \varphi.
\end{align*}
As $\eps=\eps_j\searrow 0$, the first term on the left side converges due to \eqref{conv;npcq}, as do the fifth and (after application of Lebesgue's dominated convergence theorem) the fourth integral on the right. Also convergence of the rightmost term is covered by \eqref{conv;npcq} in combination with \eqref{conv;u}. The third integral on the right converges due to Lemma \ref{lem;conv-1} and $\cep^{\f q2}\to c^{\f q2}$ in $L^2(\Om\times(0,T))$ due to \eqref{conv;c}. 
Fatou's lemma and \eqref{conv;ninL1} and \eqref{conv;c} ensure 
\[
  q\iio n^{p+1}c^{q-1} \varphi \le \liminf_{j\to\infty}  q\iio \nep^{p+1}\cep^{q-1} \varphi. 
\]
Lemma \ref{lem;conv-1} and weak lower-semicontinuity of the norm in $L^2(\Om\times(0,T))$ show that $c^{\f q2}\na n^{\f p2}\in L^2_{loc}(\Ombar\times[0,\infty);\R^N)$ and 
\begin{align*}
 \f{4(1-p)q-4q^2-p(1-p)^2\chi^2}{pq(p\chi+1-q)}&\iio c^q |\na n^{\f p2}|^2 \varphi\\ &\le \liminf_{\ep=\ep_j\to 0} \iio \frac{4(1-p)q-4q^2 -p\frac{(1-p)^2\chi^2}{(1+\ep\nep)^2}}{pq(\frac{p\chi}{1+\ep\nep}+1-q)}
       \cepq|\nabla \nep^{\frac{p}{2}}|^2\varphi.  
\end{align*}
The remaining term can be treated by an essentially similar idea (for details see \cite[Lemma 8.8]{lankeit_winkler}), so that both the second part of \eqref{def;supersol;regularity2} and the validity of \eqref{ineq;def;weak(p,q)} are ensured. 
\end{proof}
%
%
%
%
\begin{proof}[{\bf Proof of Theorem \ref{mainthm}}]
By now, Theorem \ref{mainthm} is nothing more than the combination of 
Lemmata \ref{lem;weaksol} and \ref{lem;existence} with \eqref{conv;ninmorethanL1} and \eqref{massconv;n}.
\end{proof}

\section{Acknowledgement}
T.B. and J.L. acknowledge support of the {\em Deutsche Forschungsgemeinschaft}  within the project {\em Analysis of chemotactic cross-diffusion in complex frameworks} (project no. 288366228).
M.M. is funded by JSPS Research 
Fellowships for Young Scientists (No. 17J00101). 
A major part of this work was written during joint stays at Universit\"at Paderborn and Tokyo University of Science under support from Tokyo University of Science.

%


 %
 {\footnotesize 
\def\cprime{$'$}

}
\end{document}